\documentclass[11pt,leqno]{amsart}
\usepackage{csquotes}
\newtheorem{theo}{Theorem}
\newtheorem{note}{Note}

\newcommand{\be}{\begin{equation}}
\newcommand{\ee}{\end{equation}}
\newcommand{\beas}{\begin{eqnarray*}}
\newcommand{\eeas}{\end{eqnarray*}}
\newcommand{\bea}{\begin{eqnarray}}
\newcommand{\eea}{\end{eqnarray}}

\begin{document}
\title[A simple proof of the Fundamental Theorem of Algebra]{A simple proof of the  Fundamental Theorem of Algebra}
\date{}
\author[B. Chakraborty ]{Bikash Chakraborty}
\date{}
\address{Department of Mathematics, Ramakrishna Mission Vivekananda
Centenary College, Rahara, West Bengal 700 118, India.}
\email{bikashchakraborty.math@yahoo.com, bikashchakrabortyy@gmail.com}
\maketitle
\let\thefootnote\relax
\section*{Fundamental Theorem of Algebra}
Many proofs of the Fundamental Theorem of Algebra, including various proofs based on the theory of analytic functions of a complex variable, are known. To the best of our knowledge, this proof is different from the existing ones.
\begin{theo}
Every polynomial $p(z)=z^{n}+a_{1}z^{n-1}+a_{2}z^{n-2}+\ldots+a_{n-1}z+a_{n}\in\mathbb{C}[z]$, where $n\in\mathbb{N}$, has a zero in $\mathbb{C}$.
\end{theo}
\begin{proof}
For a contradiction, let us assume that $p(z)$ has  no zero in $\mathbb{C}$. Thus the polynomials
of the sequence  $\{p_{k}(z)\}_{k=1}^{\infty}$, given by
  \beas p_{k}(z):=\frac{z \cdot p(kz)}{k^n}=z^{n+1}+\frac{a_{1}}{k}z^{n}+\frac{a_{2}}{k^2}z^{n-1}+\ldots+\frac{a_{n-1}}{k^{n-1}}z^2+\frac{a_{n}}{k^n}z,\eeas
 have only  simple zero at $0$. Next we choose an arbitrary but fixed real number $ r(>0)$. Since $z^{n+1}$ is non-zero on $\mid z\mid=r$, thus there exists a $\delta>0$ such that $\mid z^{n+1} \mid> \delta$ for every $z$ where $\mid z\mid=r$. Moreover, as $\{p_{k}(z)\}_{k=1}^{\infty}$ converges  uniformly to $z^{n+1}$ on $\mid z\mid=r$, so, there exists $N\in\mathbb{N}$ such that  $\mid p_{k}(z)\mid> \frac{\delta}{2}$ for every $k\geq N$ and for every $z$ where $\mid z\mid=r$. Thus $\left\{\frac{1}{p_{k}(z)}\right\}_{k=1}^{\infty}$ converges uniformly to $\frac{1}{z^{n+1}}$ on $\mid z\mid=r$. Hence,   \bea\label{2} \lim_{k\to\infty}\int_{\mid z\mid=r} \frac{1}{p_{k}(z)}dz=\int_{\mid z\mid=r} \frac{1}{z^{n+1}}dz.\eea
Now, using  Cauchy's integral formula, we get \bea\label{2} \lim_{k\to\infty} \frac{2\pi i\cdot k^{n}}{p(0)}=\int_{\mid z\mid=r} \frac{1}{z^{n+1}}dz,\eea
i.e., \bea\label{2} \lim_{k\to\infty}  k^{n}=0,\eea
which is impossible. Thus $p(z)$ has a zero in $\mathbb{C}$.
\end{proof}
\begin{note}
\beas\label{3} \int_{\mid z\mid=r} \frac{1}{z^{n+1}}dz=\frac{i}{r^n}\int_{0}^{2\pi}e^{-in\theta}d\theta=0. \eeas
\end{note}

\end{document}